\documentclass[reqno]{amsart}

\usepackage{amssymb,amsmath,stmaryrd,txfonts,mathrsfs,amsthm}%,undertilde}
\usepackage{comment}
\usepackage[mathscr]{euscript}
\usepackage[neveradjust]{paralist}
\usepackage{mathtools}
\usepackage{multirow}
\usepackage[outline]{contour}
\contourlength{1.2pt}
\usepackage{tikz}
\usetikzlibrary{intersections,arrows.meta,calc,quotes,math,decorations.pathreplacing,decorations.markings,cd,arrows,positioning,fit,matrix,
shapes.geometric,external,decorations.pathmorphing,backgrounds,circuits,circuits.ee.IEC,shapes}
\usepackage[a4paper,top=3cm,bottom=3cm,inner=3cm,outer=3cm]{geometry}
\usepackage[foot]{amsaddr}
\usepackage{graphicx}
\usepackage{stackrel}

\usepackage{xcolor}
\usepackage{framed,color}

\definecolor{shadecolor}{rgb}{1,0.8,0.3}
\definecolor{myurlcolor}{rgb}{0.5,0,0}
\definecolor{mycitecolor}{rgb}{0,0,0.8}
\definecolor{myrefcolor}{rgb}{0,0,0.8}
\definecolor{hyperrefcolor}{rgb}{0.5,0,0}

\usepackage{hyperref}
\hypersetup{
	colorlinks,
	linkcolor={mycitecolor},
	citecolor={mycitecolor},
	urlcolor={hyperrefcolor}
}

\usepackage[draft]{fixme}
\usepackage[capitalize]{cleveref}
\crefname{equation}{}{}
\crefname{defn}{Definition}{Definitions}
\crefname{thm}{Theorem}{Theorems}
\crefname{lem}{Lemma}{Lemmas}

\newcommand{\Int}{\raisebox{.3\depth}{$\smallint\hspace{-.01in}$}}

\theoremstyle{plain}

\newtheorem{thm}{Theorem}

\newtheorem{lem}[thm]{Lemma}

\newtheorem{cor}[thm]{Corollary}

\theoremstyle{remark}

\theoremstyle{definition}
\newtheorem{defn}[thm]{Definition}

\newcommand{\N}{\mathbb{N}}

\newcommand{\F}{\mathbb{F}}
\newcommand{\GL}{\mathrm{GL}}
\newcommand{\Ob}{\mathrm{Ob}}

\newcommand{\Set}{\mathsf{Set}}

\newcommand{\C}{\mathsf{C}}

\newcommand{\Fin}{\mathsf{Fin}}

\newcommand{\define}[1]{{\bf \boldmath{#1}}}

\title{Groupoid cardinality and random permutations}

\author{John\ C.\ Baez}
\address{School of Mathematics, University of Edinburgh, James Clerk Maxwell Building, Peter Guthrie Tait Road, Edinburgh, UK EH9 3FD}
\email{baez@math.ucr.edu}

\begin{document}

\begin{abstract}
If we treat the symmetric group \(S_n\) as a probability measure space where each element has measure \(1/n!\), then the number of cycles in a permutation becomes a random variable.  The Cycle Length Lemma describes the expected values of products of these random variables.  Here we categorify the Cycle Length Lemma by showing that it follows from an equivalence between groupoids.
\end{abstract}

\maketitle

\section{Introduction}

There is a well-behaved generalization of the concept of cardinality from finite sets to finite groupoids \cite{BD}.   But what is it good for?  As an illustration, here we use it to give a new proof of a known fact about random permutations: the Cycle Length Lemma \cite{Fo2}.   In this lemma one treats the number of \(k\)-cycles in a  permutation of \(n\) things as a random variable, where each permutation occurs with equal probability.   The lemma says that in the limit as \(n \to \infty\), this random variable approaches a Poisson distribution with mean \(1/k\).   Furthermore, in the \(n \to \infty\) limit these random variables become independent for different choices of \(k\).

These are quick rough statements.  In Section \ref{sec:cycle} we state the Cycle Length Lemma in a precise
way.   In Section \ref{sec:categorified} we prove a \emph{categorified} version of the Cycle Length Lemma, which asserts an equivalence of groupoids.  In Section \ref{sec:groupoid} we derive the original version of the lemma from this categorified version by taking the cardinalities of these groupoids.  The categorified version contains more information, so it is not just a trick for proving the original lemma (which is, after all, quite easy to show).  Instead, it reveals the original lemma as a consequence of a stronger fact about groupoids.  

In Section \ref{sec:conclusion} we sketch how some of the ideas here generalize to other finite groups.   

\section{The Cycle Length Lemma}
\label{sec:cycle}

In the theory of random permutations, we treat the symmetric group \(S_n\) as a probability measure space where each element has the same measure, namely \(1/n!\).   Functions \(f \colon S_n \to \mathbb{R}\) then become random variables, and we can study their expected values:
\[ E(f) = \frac{1}{n!} \sum_{\sigma \in S_n} f(\sigma). \]
An important example is the function 
\[ c_k \colon S_n \to \mathbb{N} \]
that counts, for any permutation \(\sigma \in S_n\), its number of cycles of length \(k\), also called \(k\)-cycles.   A well-known but striking fact about random permutations is that whenever \(k \le n\), the expected number of \(k\)-cycles is \(1/k\):
\[ E(c_k) = \frac{1}{k}. \]
This has some nice consequences.  For example, a randomly chosen permutation of any finite set has, on average, one fixed point.  Also, its expected number of cycles is
\[    1 + \frac{1}{2} + \cdots + \frac{1}{n}, \]
which for large \(n\) becomes close to \(\ln n\) plus Euler's constant \(\gamma\).

Another striking fact is that whenever \(j \ne k\) and \(j + k \le n\), so that it's possible for a permutation \(\sigma \in S_n\) to have both a \(j\)-cycle and a \(k\)-cycle, the random variables \(c_j\) and \(c_k\) are uncorrelated in the following sense:
\[ E(c_j c_k) = E(c_j) E(c_k) . \]
You might at first think that having many \(j\)-cycles for some large \(j\) would tend to inhibit the presence of \(k\)-cycles for some other large value of \(k\), but that is not true unless \(j + k > n\), when it suddenly becomes \emph{impossible} to have both a \(j\)-cycle and a \(k\)-cycle!

These two facts are special cases of the Cycle Length Lemma.  To state this lemma in full generality, recall that the number of ordered \(p\)-tuples of distinct elements of an \(n\)-element set is the \define{falling power}
\[ x^{\underline{p}} = x(x-1)(x-2) \, \cdots \, (x-p+1). \]
It follows that the function
\[ c_k^{\underline{p}} \colon S_n \to \mathbb{N} \]
counts, for any permutation in \(S_n\), its ordered \(p\)-tuples of distinct \(k\)-cycles.  We can also replace the word `distinct' here by `disjoint', without changing the meaning, since distinct cycles must be disjoint. 

The two striking facts mentioned above generalize as follows:

\begin{enumerate}
\item First, whenever \(p k \le n\), so that it is \emph{possible} for a permutation in \(S_n\) to have \(p\) distinct \(k\)-cycles, then
\[  E(c_k^{\underline{p}}) = \frac{1}{k^p}. \]
For readers familiar with the moments of a Poisson distribution, here is a nice equivalent way to state this equation: when  \(p k \le n\), the \(p\)th moment of the random variable \(c_k\) equals that of a Poisson distribution with mean \(1/k\).
\item Second, as \(n \to \infty\) the random variables \(c_k\) become better and better approximated by independent Poisson distributions.   To state this precisely we need a bit of notation.  Let \(\vec{p}\) denote an \(n\)-tuple \((p_1 , \dots, p_n)\) of natural numbers, and let
\[  |\vec{p}|  = p_1 + 2p_2 + \cdots + n p_n. \]
If  \(|\vec{p}| \le n\), it is possible for a permutation \(\sigma \in S_n\) to have a collection of distinct cycles, with \(p_1\) cycles of length 1, \(p_2\) cycles of length 2, and so on up to \(p_n\) cycles of length \(n\).   If \(|\vec{p}| > n\), this is impossible.  In the former case, where \(|\vec{p}| \le n\), we always have
\[ E\left( \prod_{k=1}^n c_k^{\underline{p}_k} \right) = \prod_{k=1}^n E( c_k^{\underline{p}_k}) . \]
\end{enumerate}

Taken together, 1) and 2) are equivalent to the Cycle Length Lemma, which may be stated in a unified way as follows:
\vskip 1em
\textbf{The Cycle Length Lemma}.  Suppose \(p_1 , \dots, p_n \in \mathbb{N}\).  Then
\[ E\left( \prod_{k=1}^n c_k^{\underline{p}_k} \right) 
=  \left\{ \begin{array}{ccc}  \displaystyle{ \prod_{k=1}^n \frac{1}{k^{p_k}} } & &  \mathrm{if} \; |\vec{p}| \le n \\ \\
0 & & \mathrm{if} \; |\vec{p}| > n. \end{array} \right. \]

This appears, for example, in Ford's comprehensive review of the statistics of cycle lengths in random permutations \cite[Lem.\ 3.1]{Fo2}.  He attributes it to Watterson \cite[Thm.\ 7]{W}.  The most famous special case is when \(|\vec{p}| = n\), which apparently goes back to Cauchy.

For more details on the sense in which random variables \(c_k\) approach independent Poisson distributions, see Arratia and Tavar\'e \cite{AT}.

\section{The Categorified Cycle Length Lemma}
\label{sec:categorified}

To categorify the Cycle Length Lemma, the key is to treat a permutation as an extra structure that we can put on a set, and then consider the groupoid of \(n\)-element sets equipped with this extra structure:

\begin{defn} Let  \(\mathsf{Perm}_n\) be the groupoid in which 
\begin{itemize}
\item an object is an \(n\)-element set equipped with a permutation \(\sigma \colon X \to X\)
\end{itemize}
and 
\begin{itemize}
\item a morphism from \(\sigma \colon X \to X\)  to \(\sigma' \colon X' \to X'\) is a bijection \(f \colon X \to X'\) that is \define{permutation-preserving} in the following sense: 
\[ f \circ \sigma \circ f^{-1} = \sigma'. \]
\end{itemize}
\end{defn}

We'll need the following strange fact below: if \(n < 0\) then \(\mathsf{Perm}_n\) is the empty groupoid (that is, the groupoid with no objects and no morphisms).

More importantly, we'll need a fancier groupoid where a set is equipped with a permutation together with a list of distinct cycles of specified lengths.  For any \(n \in \mathbb{N}\) and any \(n\)-tuple of natural numbers \(\vec{p} = (p_1 , \dots, p_n)\), recall that we have defined
\[  |\vec{p}| = p_1 + 2p_2 + \cdots + n p_n. \]

\begin{defn}
Let \(\C_{\vec{p}}\) be the groupoid of \(n\)-element sets \(X\) equipped with a permutation \(\sigma \colon X \to X\) that is in turn equipped with a choice of an ordered \(p_1\)-tuple of distinct \(1\)-cycles, an ordered \(p_2\)-tuple of distinct \(2\)-cycles, and so on up to an ordered \(p_n\)-tuple of distinct \(n\)-cycles.  A morphism in this groupoid is a bijection that is permutation-preserving and also preserves the ordered tuples of distinct cycles.  
\end{defn}

Note that if \(|\vec{p}| > n\), no choice of disjoint cycles with the specified property exists, so \(\C_{\vec{p}}\) is the empty groupoid.   

Finally, we need a bit of standard notation.  For any group \(G\) we write \(\mathsf{B}(G)\) for its \define{delooping}: that is, the groupoid that has one object \(\star\) and \(\mathrm{Aut}(\star) = G\).

\begin{thm}
{\bf (The Categorified Cycle Length Lemma.)} For any \(\vec{p} = (p_1 , \dots, p_n) \in \mathbb{N}^n\) we have   
\[ \C_{\vec{p}} \simeq \mathsf{Perm}_{n - |\vec{p}|} \; \times \; \prod_{k = 1}^n \mathsf{B}(\mathbb{Z}/k)^{p_k}. \]
\end{thm}

\begin{proof}
Both sides are empty groupoids when \(|\vec{p}| > n\), so assume \(|\vec{p}| \le n\).   A groupoid is equivalent to any full subcategory of that groupoid containing at least one object from each isomorphism class.   So, fix an \(n\)-element set \(X\) and a subset \(Y \subseteq X\) with \(n - |\vec{p}|\) elements.   Partition \(X - Y\) into subsets \(S_{k\ell}\) where \(S_{k \ell}\) has cardinality \(k\), \(1 \le k \le n\), and \(1 \le \ell \le p_k\).  Every object of \(\C_{\vec{p}}\) is isomorphic to the chosen set \(X\) equipped with some permutation \(\sigma \colon X \to X\) that has each subset \(S_{k \ell}\) as a \(k\)-cycle.  Thus \(\C_{\vec{p}}\) is equivalent to its full subcategory containing only objects of this form.

An object of this form consists of an arbitrary permutation \(\sigma_Y \colon Y \to Y\) and a cyclic permutation \(\sigma_{k \ell} \colon S_{k \ell} \to S_{k \ell}\) for each \(k,\ell\) as above.   Consider a second object of this form, say \(\sigma'_Y \colon Y \to Y\) equipped with cyclic permutations \(\sigma'_{k \ell}\).  Then a morphism from the first object to the second consists of two pieces of data.  First, a bijection
\[ f \colon Y \to Y \]
such that
\[ \sigma'_Y = f \circ \sigma_Y \circ f^{-1}. \]
Second, for each \(k,\ell\) as above, a bijection
\[ f_{k \ell} \colon S_{k \ell} \to S_{k \ell} \]
such that
\[ \sigma'_{k \ell} = f_{k \ell} \circ \sigma_{k \ell} \circ f_{k \ell}^{-1}. \]

Since \(Y\) has \(n - |\vec{p}|\) elements, while \(\sigma_{k \ell}\) and \(\sigma'_{k \ell}\) are cyclic permutations of \(k\)-element sets, it follows that \(\C_{\vec{p}}\) is equivalent to 
\[ \mathsf{Perm}_{n - |\vec{p}|} \; \times \; \prod_{k = 1}^n \mathsf{B}(\mathbb{Z}/k)^{p_k}. \qedhere \]
\end{proof}

The case where \(|\vec{p}| = n \) is especially pretty, since then our chosen cycles completely fill up our \(n\)-element set and we have
\[ \C_{\vec{p}} \simeq \prod_{k = 1}^n \mathsf{B}(\mathbb{Z}/k)^{p_k}.  \]

\section{Groupoid Cardinality}
\label{sec:groupoid}

The cardinality of finite sets has a natural extension to finite groupoids, which turns out to be the key to extracting results on random permutations from category theory.   We briefly recall this concept \cite{BD}.  Any finite groupoid \(\mathsf{G}\) is equivalent to a coproduct of finitely many one-object groupoids, which are deloopings of finite groups \(G_1, \dots, G_m\):
\[ \mathsf{G} \simeq \sum_{i = 1}^m \mathsf{B}(G_i), \]
and then the \define{cardinality} of \(\mathsf{G}\) is defined to be
\[ |\mathsf{G}| = \sum_{i = 1}^m \frac{1}{|G_i|}. \]
This concept of groupoid cardinality has various nice properties.  For example it is additive:
\[ |\mathsf{G} + \mathsf{H}| = |\mathsf{G}| + |\mathsf{H}| \]
and multiplicative:
\[ |\mathsf{G} \times \mathsf{H}| = |\mathsf{G}| \times |\mathsf{H}| \]
and invariant under equivalence of groupoids:
\[ \mathsf{G} \simeq \mathsf{H} \implies |\mathsf{G}| = |\mathsf{H}|. \]

None of these three properties forces us to define \(|\mathsf{G}|\) as the sum of the \emph{reciprocals} of the cardinalities \(|G_i|\): any other power of these cardinalities would work just as well.  What makes the reciprocal cardinalities special is that if \(G\) is a finite group acting on a set \(S\), we have 
\[ |S\sslash G| = |S|/|G| \]
where the groupoid \(S \sslash G\) is the \define{weak quotient} or \define{homotopy quotient} of \(S\) by \(G\), also called the \define{action groupoid}.  This is the groupoid with elements of \(S\) as objects and one morphism from \(s\) to \(s'\) for each \(g \in G\) with \(g s = s'\), with composition of morphisms coming from multiplication in \(G\).  

The groupoid of \(n\)-element sets equipped with permutation, \(\mathsf{Perm}_n\), has a nice description in terms of weak quotients:

\begin{lem} 
\label{lem:Perm}
For all \(n \in \mathbb{N}\) we have an equivalence of groupoids
\[ \mathsf{Perm}_n \simeq S_n \sslash S_n \]
where the group \(S_n\) acts on the underlying set of \(S_n\) by conjugation.
\end{lem}

\begin{proof}
We use the fact that \(\mathrm{Perm}_n\) is equivalent to any full subcategory of  \(\mathrm{Perm}_n\) containing at least one object from each isomorphism class.  For \(\mathsf{Perm}_n\) we can get such a subcategory by fixing an \(n\)-element set, say \(X = \{1,\dots, n\}\), and taking only objects of the form \(\sigma \colon X \to X\), i.e. \(\sigma \in S_n\).  A morphism from \(\sigma \in S_n\) to \(\sigma' \in S_n\) is then a permutation \(\tau \in S_n\) such that
\[ \sigma' = \tau  \sigma \tau^{-1} .\]
But this subcategory is precisely \(S_n \sslash S_n\).  
\end{proof}

\begin{cor} For all \(n \in \mathbb{N}\) we have
\[ |\mathrm{Perm}_n| = 1. \]
\end{cor}

\begin{proof}  We have \(|\mathrm{Perm}_n| = |S_n \sslash S_n| = |S_n|/|S_n| = 1\).  
\end{proof}

For another useful perspective, note that for any finite groupoid \(\mathsf{G}\) we have
\[   |\mathsf{G}| = \sum_{g \in \mathsf{G}} \frac{1}{|\mathrm{out}(g)|}   \]
where \(\mathrm{out}(g)\) is the set of all morphisms out of the object \(g \in \mathsf{G}\).    The groupoid \(S_n \sslash S_n\) has one object for each \(\sigma \in S_n\), and there are \(n!\) morphisms out of each object, so the groupoid cardinality of \(S_n \sslash S_n\) is \(1\).   

This clarifies why we can prove results on random permutations using the groupoid \(\mathrm{Perm}_n\): this groupoid is equivalent to \(S_n \sslash S_n\), which has one object for each permutation in \(S_n\), each contributing \(1/n!\) to the groupoid cardinality.

Now let us use these ideas to derive the original Cycle Length Lemma from the categorified version.

\begin{thm} {\bf (The Cycle Length Lemma.)} 
\label{thm:cycle_length_lemma}
 Suppose \(p_1 , \dots, p_n \in \mathbb{N}\).  Then
\[ E\left( \prod_{k=1}^n C_k^{\underline{p}_k} \right) 
=  \left\{ \begin{array}{ccc}  \displaystyle{ \prod_{k=1}^n \frac{1}{k^{p_k}} } & &  \mathrm{if} \; |\vec{p}| \le n \\ \\
0 & & \mathrm{if} \; |\vec{p}| > n. \end{array} \right. \]
\end{thm}

\begin{proof} 
We know that
\[ \C_{\vec{p}} \simeq \mathsf{Perm}_{n - |\vec{p}|} \; \times \; \prod_{k = 1}^n \mathsf{B}(\mathbb{Z}/k)^{p_k}. \]
So, to prove the Cycle Length Lemma it suffices to show three things:
\[  |\C_{\vec{p}}| = E\left( \prod_{k=1}^n c_k^{\underline{p}_k} \right) \]
\[ \mathsf{Perm}_{n - |\vec{p}|} = \left\{  \begin{array}{ccc} 1 & & \mathrm{if} \; |\vec{p}| \le n \\ \\
0 & &  \mathrm{if} \; |\vec{p}| > n \end{array} \right. \]
and
\[ |\mathsf{B}(\mathbb{Z}/k)| = 1/k. \]

The last of these is immediate from the definition of groupoid cardinality.  The second follows from the Corollary above, together with the fact that \(\mathsf{Perm}_{n - |\vec{p}|}\) is the empty groupoid when \(|\vec{p}| >  n\).  Thus we are left needing to show that 
\[  |\C_{\vec{p}}| = E\left( \prod_{k=1}^n c_k^{\underline{p}_k} \right). \]
We prove this by computing the cardinality of a groupoid equivalent to \(\C_{\vec p}\).  We claim this groupoid is of the form \(Q_{\vec{p}} \sslash S_n\) where \(Q_{\vec{p}}\) is some set on which \(S_n\) acts.   As a result we have
\[ |\C_{\vec{p}}| = |Q_{\vec{p}} \sslash S_n| = |Q_{\vec{p}}| / n!  \]
and to finish the proof we need to show
\[  E\left( \prod_{k=1}^n c_k^{\underline{p}_k} \right) =  |Q_{\vec{p}}| / n!\, .  \]

What is the set \(Q_{\vec{p}}\), and how does \(S_n\) act on it?   An element of \(Q_{\vec{p}}\) is a permutation \(\sigma \in S_n\) equipped with an ordered \(p_1\)-tuple of distinct \(1\)-cycles, an ordered \(p_2\)-tuple of distinct \(2\)-cycles, and so on up to an ordered \(p_n\)-tuple of distinct \(n\)-cycles.   Any element \(\tau \in S_n\) acts on \(Q_{\vec{p}}\) in a natural way, by conjugating the permutation \(\sigma \in S_n\) to obtain a new permutation, and mapping the chosen cycles of \(\sigma\) to the corresponding cycles of this new conjugated permutation \(\tau \sigma \tau^{-1}\).

Recalling the definition of the groupoid \(\C_{\vec{p}}\), it is clear that any element of \(Q_{\vec{p}}\) gives an object of  \(\C_{\vec{p}}\), and any object is isomorphic to one of this form.   Furthermore any permutation \(\tau \in S_n\) gives a morphism between such objects, all morphisms between such objects are of this form, and composition of these morphisms is just multiplication in \(S_n\).  It follows that
\[ \C_{\vec{p}} \simeq Q_{\vec{p}} \sslash S_n. \]

To finish the proof, note that
\[ E\left( \prod_{k=1}^n c_k^{\underline{p}_k} \right) \]
is \(1/n!\) times the number of ways of choosing a permutation \(\sigma \in S_n\) and equipping it with an ordered \(p_1\)-tuple of distinct \(1\)-cycles, an ordered \(p_2\)-tuple of distinct \(2\)-cycles, and so on.   This is the same as \( |Q_{\vec{p}}| / n!\).  
\end{proof}

\section{Conclusion}
\label{sec:conclusion}

We have opted to treat an example rather than develop a general theory, but many of the ideas here go beyond the symmetric group.   Any finite group \(G\) acts on itself by conjugation and gives a groupoid \(G \sslash G\) of cardinality \(1\).   Any functor \(F \colon G \sslash G \to \Fin\Set\) describes a conjugation-equivariant structure we can put on elements of \(G\), with \(F(g)\) being the set of structures we can put on the element \(g \in G\).   Taking the ordinary cardinality of these sets, we obtain a function \(|F| \colon G \to \N\).   Its expected value with respect to the normalized Haar measure on \(G\) is, by definition,
\[  E(|F|) = \frac{1}{|G|} \sum_{g \in G} |F(g)|   .\]
However, \(E(|F|)\) also equals the cardinality of a certain groupoid for which an object is an element \(g \in G\)  equipped with a structure \(x \in F(g)\).  This groupoid is the familiar \define{category of elements} of  \(F\), denoted \(\Int F\), for which:
\begin{enumerate}
\item an object is a pair \((g,x)\) where \(g \in G\) and \(x \in F(g) \);
\item a morphism from \((g,x)\) to \((g',x')\) is an element \(h \in G\) such that \(g' = h g h^{-1}\) and
\(x' = F(h)(x)\);
\item composition of morphisms is multiplication of group elements.
\end{enumerate}

\begin{thm} 
\label{thm:general}
If \(G\) is a finite group and \(F \colon G \sslash G \to \Fin\Set\) is a functor, then 
\[           E(|F|) = \left|\Int F \right|. \]
\end{thm}

\begin{proof}
Let \(\Ob(\Int F)\) be the set of objects of \(\Int F\).  The group \(G\) acts on \(\Ob(\Int F)\), with \(h \in G\)
mapping the object \((g,x)\) to the object \((hgh^{-1}, F(h)(x))\).   Using the explicit description of \(\Int F\) in items (1)--(3) above, there is an evident isomorphism of groupoids
\[         \Int F \cong \Ob(\Int F) \sslash G  \]
that is the identity on objects and sends each morphism \(h\) from \((g,x)\) to \((g',x')\) to the analogous
morphism in \(\Ob(\Int F) \sslash G  \).    It follows that
\[    \left|\Int F \right| = \left| \Ob(\Int F) \sslash G  \right| = |\Ob(\Int F)|/|G| = 
\displaystyle{ \frac{1}{|G|} \sum_{g \in G} |F(g)| }  = E(|F|) . \qedhere\]
\end{proof}

The expected value of \(|F|\) is its integral over \(G\) with respect to normalized Haar measure, so we can write it as \(\Int \, |F|\), and then the theorem above takes an amusing though perhaps confusing form:
\[             \Int \, |F| = \left| \Int F \right|. \]

The above theorem sheds new light on the proof of Theorem \ref{thm:cycle_length_lemma}, because the \(S_n\)-set \(Q_{\vec{p}}\) in that proof is none other than \(\Ob(\Int C_{\vec p})\) for the functor \(C_{\vec{p}} \colon S_n \sslash S_n \to \Fin\Set \) assigning to any permutation the set where an element is an ordered \(p_1\)-tuple of distinct \(1\)-cycles, an ordered \(p_2\)-tuple of distinct \(2\)-cycles, and so on.   Thus, the groupoid \(\C_{\vec{p}}\) is equivalent to \(\Int C_{\vec{p}}\).   The same ideas apply to other structures that we can put on a finite set equipped with a permutation.

The above theorem may also let us derive results about random elements of other groups from equivalences of groupoids.   Results on \(\GL(n,\F_q)\) are promising candidates \cite{Fu}, since some are already proved using generating functions, which are connected to the category-theoretic techniques used here \cite{BD,BLL,J}, and there are powerful analogies between finite sets and finite-dimensional vector spaces over finite fields \cite{Le,Lo}.


\begin{thebibliography}{5}

\bibitem{AT} Richard Arratia and Simon Tavar\'e, The cycle structure of random permutations, \textsl{The Annals of Probability} \textbf{20} (1992), 1567--1591.  Available at \href{https://doi.org/10.1214/aop/1176989707}{https://doi.org/10.1214/aop/1176989707}.

\bibitem{BD} John C.\ Baez and James Dolan, From finite sets to Feynman diagrams, in \textsl{Mathematics Unlimited---2001 and Beyond}, vol. 1, eds. Bj\"orn Engquist and Wilfried Schmid, Springer, Berlin, 2001, pp.\ 29--50.   Available as \href{http://arxiv.org/abs/math.QA/0004133}{arXiv:0004133}.

%\bibitem{BHW} John C.\ Baez, Alexander E.\ Hoffnung and Christopher D.\ Walker, Higher-dimensional algebra VII: groupoidification, \textsl{Theory and Applications of Categories} \textbf{24} (2010), 489--553.   Available as \href{http://arxiv.org/abs/0908.4305}{arXiv:0908.4305}.

\bibitem{BLL} Fran\c cois Bergeron, Gilbert Labelle and Pierre Leroux, \textsl{Combinatorial Species and Tree-like Structures}, Cambridge U.\ Press, Cambridge, 1998.

%\bibitem{Fo1} Kevin Ford, \textsl{Anatomy of Integers and Random Permutations---Course Lecture Notes}.   %Available at
%\hfill \break \href{https://faculty.math.illinois.edu/~ford/Anatomy\_lectnotes.pdf}{https://faculty.math.illinois.edu/$\sim$ford/Anatomy\_lectnotes.pdf}.

\bibitem{Fo2} Kevin Ford, Cycle type of random permutations: a toolkit, \textsl{Discrete Analysis} \textbf{29} (2022).  Available as \href{https://arxiv.org/abs/2104.12019}{ arXiv:2104.12019}.

\bibitem{Fu} Jason Fulman, Random matrix theory over finite fields: a survey, \textsl{Bulletin of the American Mathematical Society} \textbf{39} (2001), 51--85. Available as \href{https://arxiv.org/abs/math/0003195}{arXiv:0003195}.

\bibitem{J} Andr\'e Joyal, Une th\'eorie combinatoire des s\'eries formelles, \textsl{Advances in Mathematics}
\textbf{42} (1981), 1--82.

\bibitem{Le} Tom Leinster, The probability that an operator is nilpotent, \textsl{American Mathematical Monthly} \textbf{128} (2021), 371--375.  Available as \href{https://arxiv.org/abs/1912.12562}{arXiv:1912.12562}.

\bibitem{Lo}  Oliver Lorscheid, Algebraic groups over the field with one element, \textsl{Mathematische Zeitschrift} \textbf{271} (2012), 117--138.  Available as \href{https://arxiv.org/abs/0907.3824}{arXiv:0907.3824}.

\bibitem{W} G.\ A.\ Watterson, The sampling theory of selectively neutral alleles, \textsl{Advances in Applied Probability} \textbf{6} (1974), 463--488.

\end{thebibliography}
\end{document}